\documentclass[11pt]{article}

\usepackage[a4paper,margin=30mm]{geometry}
\usepackage{amsmath,amssymb,amsthm,mathtools}
\usepackage{mathrsfs}
\usepackage{bm}
\usepackage{enumitem}
\usepackage{microtype}
\usepackage[hidelinks]{hyperref}

\newtheorem{theorem}{Theorem}[section]
\newtheorem{proposition}[theorem]{Proposition}
\newtheorem{corollary}[theorem]{Corollary}
\newtheorem{lemma}[theorem]{Lemma}

\theoremstyle{definition}
\newtheorem{definition}[theorem]{Definition}

\theoremstyle{remark}
\newtheorem{remark}[theorem]{Remark}

\newcommand{\PP}{\mathbb P}
\newcommand{\CC}{\mathbb C}
\newcommand{\cZ}{\mathcal Z}
\newcommand{\cH}{\mathcal H}

\newcommand{\Div}{\operatorname{Div}}
\newcommand{\ord}{\operatorname{ord}}
\newcommand{\Hom}{\operatorname{Hom}}

\title{\bfseries
The Relative Differential Morphism and the Logarithmic Connection\\
of Dual Conformal Nets
}

\author{A.\,G. Nuramatov}

\date{}

\begin{document}

\maketitle

\begin{abstract}
Let \(X\) be a compact connected Riemann surface equipped with two
nonconstant meromorphic projections
\[
q,z:X\longrightarrow \PP^1.
\]
Their differentials determine a canonical meromorphic relative
differential morphism
\[
\cZ:T_q:=q^*T\PP^1\longrightarrow T_z:=z^*T\PP^1,
\qquad
\cZ=Tz\circ(Tq)^{-1},
\]
whose local coefficient in coordinate frames is
\(z_q=dz/dq\).  Its divisor is \(R_z-R_q\), and its local
logarithmic derivatives define the standard rank-one logarithmic
connection on
\(\cH=\Hom(T_q,T_z)\).  The general passage from a meromorphic
section to its logarithmic derivative is classical.  The main
point of the present paper is that, in the dual conformal-net
setting, the Levi--Civita curvature data of the two nets realize
this logarithmic connection by the symmetric formula
\[
\frac1H\widetilde\omega\,dz-H\omega\,dq
=
2i\,d\log z_q
=
2i\,\cZ^{-1}d\cZ,
\]
where the last expression is understood through local
trivializations of \(\cH\).  Thus the two dual conformal
connection differentials occur with opposite signs and their
difference is the logarithmic Maurer--Cartan form of the relative
differential morphism.  Its residues recover the relative
ramification multiplicities,
\[
\operatorname{Res}_p\Theta=2i\bigl(r_z(p)-r_q(p)\bigr),
\]
and globally
\[
\sum_{p\in X}\operatorname{Res}_p\Theta
=4i\bigl(\deg z-\deg q\bigr).
\]
Thus local conformal connection data determine both relative
ramification and the difference of the degrees of the two
projections.  We also discuss the dual morphism, square-root
issue, and a hyperelliptic example.
\end{abstract}

\medskip

\noindent
\textbf{Keywords.}
Riemann surface; meromorphic projection; pullback tangent bundle;
solder form; logarithmic connection; Maurer--Cartan form;
ramification divisor.

\medskip

\noindent
\textbf{MSC 2020.}
53C05, 30F10, 14H55.

\section{Introduction}

The Maurer--Cartan form is one of the basic constructions of
differential geometry.  If \(g\) is a map into a Lie group \(G\),
then
\[
g^{-1}dg
\]
measures the infinitesimal variation of \(g\) after translation
back to the identity.  The construction is logarithmic in nature:
in the one-dimensional multiplicative group \(\CC^*\) it reduces
simply to
\[
g^{-1}dg=d\log g.
\]

The purpose of the present paper is to identify an analogous
construction associated with two meromorphic projections on a
Riemann surface.

The general theory of regular singular and logarithmic
connections is classical and goes back to Deligne
\cite{Deligne}; for rank-one logarithmic connections on compact
Riemann surfaces, see also \cite{Singh}.  In particular, when a
nonzero meromorphic section \(s\) of a line bundle is written
locally as \(s=s_\alpha E_\alpha\), its logarithmic derivatives
\(d\log s_\alpha\) give local meromorphic connection forms (up to
the sign convention for the covariant derivative).  Related uses
of logarithmic differentials of meromorphic sections appear, for
example, in \cite{BenZviBiswas}.  Thus the general operation
\(s\mapsto d\log s\) is standard; the point of the present paper
is its application to the relative morphism of two projections
and its identification with the differential-geometric form
arising from dual conformal coordinate nets.

Let \(X\) be a compact connected Riemann surface and let
\[
q,z:X\longrightarrow\PP^1
\]
be nonconstant meromorphic maps.  Each projection determines a
pullback tangent line bundle,
\[
L_q:=q^*T\PP^1,
\qquad
L_z:=z^*T\PP^1.
\]
Their differentials provide two ways of identifying the tangent
line of \(X\) with a tangent line of the target sphere, at least
away from ramification.  Comparing these two identifications
produces a canonical meromorphic morphism
\[
\cZ:L_q\longrightarrow L_z.
\]

In local coordinates this morphism takes the familiar form
\[
\cZ
=
z_q\,\frac{\partial}{\partial z}\otimes dq,
\qquad
z_q=\frac{dz}{dq}.
\]
The coefficient \(z_q\) alone is coordinate dependent.  The
tensorial expression above, however, is intrinsic.  It is the
relative differential transformation between the two pulled-back
tangent bundles.

The logarithmic derivative of \(\cZ\) is therefore not, by
itself, a new construction: locally, if
\[
\cZ=g_\alpha E_\alpha,
\]
then the usual rank-one logarithmic connection is represented by
\[
\alpha_\alpha=g_\alpha^{-1}dg_\alpha.
\]
We use the normalization
\[
\Theta_\alpha=2i\,\alpha_\alpha,
\]
and write symbolically
\[
\Theta=2i\,\cZ^{-1}d\cZ.
\]
This notation always refers to the corresponding collection of
local logarithmic connection forms, not to the ordinary exterior
derivative of a globally defined scalar function.

The new geometric point is obtained when the pair \((q,z)\)
comes from dual conformal coordinate nets.  With the curvature
coefficients \(\omega\) and \(\widetilde\omega\) and conformal
factor \(H\) used in \cite{NuramatovAC}, we prove
\[
\frac1H\widetilde\omega\,dz-H\omega\,dq
=
2i\,d\log z_q
=
2i\,\cZ^{-1}d\cZ.
\]
Thus the logarithmic connection of the intrinsic relative
differential morphism is exactly the difference of the two dual
Levi--Civita connection differentials.  This is the main theorem
of the paper.

The paper is organized as follows.  Sections~2--4 construct the
relative differential morphism and describe its divisor in terms
of ramification.  Sections~5--6 recall the associated rank-one
logarithmic connection and its local Maurer--Cartan description.
Sections~7--9 discuss flatness, residues, the dual morphism,
square roots, and a hyperelliptic example.  Section~10 proves the
main conformal-net identification, and the final sections collect
its geometric consequences.

\section{Two meromorphic projections and their pullback bundles}

Let \(X\) be a compact connected Riemann surface and
\[
q,z:X\longrightarrow\PP^1
\]
two nonconstant meromorphic maps.

Associated with them are the holomorphic line bundles
\[
L_q=q^*T\PP^1,
\qquad
L_z=z^*T\PP^1.
\]
For \(p\in X\), their fibers are
\[
(L_q)_p=T_{q(p)}\PP^1,
\qquad
(L_z)_p=T_{z(p)}\PP^1.
\]

The differentials of the two maps are bundle morphisms
\[
Tq:TX\longrightarrow L_q,
\qquad
Tz:TX\longrightarrow L_z.
\]
At a point where \(dq\neq0\), the first morphism is an isomorphism
of complex lines; similarly \(Tz\) is an isomorphism wherever
\(dz\neq0\).

Let
\[
R_q=\sum_{p\in X}r_q(p)\,p,
\qquad
R_z=\sum_{p\in X}r_z(p)\,p
\]
denote the ramification divisors of \(q\) and \(z\),
respectively.  Thus \(r_q(p)\) and \(r_z(p)\) are the vanishing
orders of \(Tq\) and \(Tz\).

Set
\[
X^\circ
=
X\setminus
\bigl(\operatorname{supp}R_q\cup\operatorname{supp}R_z\bigr).
\]
On \(X^\circ\), both \(Tq\) and \(Tz\) are bundle isomorphisms.

The basic comparison map is therefore already forced upon us:
\[
Tz\circ(Tq)^{-1}:L_q|_{X^\circ}\longrightarrow L_z|_{X^\circ}.
\]

\begin{definition}
On \(X^\circ\), define the canonical relative differential
morphism by
\[
\cZ:=Tz\circ(Tq)^{-1}.
\]
\end{definition}

The morphism extends meromorphically across the omitted points,
as will be made explicit below.

\section{Solder forms and the relative morphism}

The identity endomorphism of \(T\PP^1\) has, in a local coordinate
\(w\), the familiar tensorial representation
\[
dw\otimes\frac{\partial}{\partial w}.
\]
Indeed,
\[
\left(
dw\otimes\frac{\partial}{\partial w}
\right)
\left(
a\frac{\partial}{\partial w}
\right)
=
a\frac{\partial}{\partial w}.
\]
This expression is independent of the choice of the coordinate
\(w\).  It is the canonical solder tensor of the tangent line
bundle.

Pulling it back by \(q\) and \(z\) gives the two local solder
expressions
\[
dq\otimes\frac{\partial}{\partial q},
\qquad
dz\otimes\frac{\partial}{\partial z}.
\]

We now compare them.

Let \(U\subset X^\circ\) be a coordinate neighborhood.  Since
\(dq\neq0\) on \(U\), \(q\) itself may be used as a local
coordinate, and \(z\) is locally a holomorphic function of \(q\).
Write
\[
z_q=\frac{dz}{dq}.
\]
Then
\[
dz=z_q\,dq.
\]

\begin{proposition}
On \(U\subset X^\circ\), the relative morphism \(\cZ\) is
represented by
\[
\cZ
=
z_q\,
\frac{\partial}{\partial z}\otimes dq.
\]
\end{proposition}

\begin{proof}
A vector in \(L_q\) has the local form
\[
v=a\,\frac{\partial}{\partial q}.
\]
Since
\[
(Tq)^{-1}
\left(
a\,\frac{\partial}{\partial q}
\right)
=
a\,\frac{\partial}{\partial q}
\]
in the \(q\)-coordinate on \(X\), applying \(Tz\) gives
\[
Tz
\left(
a\,\frac{\partial}{\partial q}
\right)
=
a\,z_q\,\frac{\partial}{\partial z}.
\]
Thus
\[
\cZ(v)
=
z_q\,dq(v)\frac{\partial}{\partial z},
\]
which is exactly the asserted tensorial expression.
\end{proof}

The tensorial formulation is essential.  The scalar \(z_q\) is
not itself invariant under independent changes of coordinates on
the two target spheres, whereas the bundle morphism \(\cZ\) is.

Suppose
\[
Q=Q(q),
\qquad
Z=Z(z)
\]
are new local coordinates.  Then
\[
\frac{dZ}{dQ}
=
\frac{Z_z}{Q_q}\,z_q,
\]
while
\[
\frac{\partial}{\partial Z}
=
\frac1{Z_z}\frac{\partial}{\partial z},
\qquad
dQ=Q_q\,dq.
\]
Hence
\[
\frac{dZ}{dQ}
\frac{\partial}{\partial Z}\otimes dQ
=
z_q
\frac{\partial}{\partial z}\otimes dq.
\]

We have therefore proved the following.

\begin{proposition}
The expression
\[
z_q\,
\frac{\partial}{\partial z}\otimes dq
\]
defines an intrinsic meromorphic section of
\[
\Hom(L_q,L_z)
\simeq
L_z\otimes L_q^{-1}.
\]
\end{proposition}

This canonical section is the basic geometric object of the
paper.

\section{Ramification and the divisor of the relative morphism}

The relative morphism is nonsingular on \(X^\circ\), but its
meromorphic continuation records the difference between the
ramification behaviors of the two projections.

Let \(p\in X\), and choose a local coordinate \(t\) centered at
\(p\), together with holomorphic coordinates \(u\) and \(v\) on
the two target spheres centered at \(q(p)\) and \(z(p)\),
respectively.  Suppose
\[
u\circ q=a\,t^{m+1}+O(t^{m+2}),
\qquad
v\circ z=b\,t^{n+1}+O(t^{n+2}),
\]
with \(a,b\neq0\).  Then
\[
d(u\circ q)
=
a(m+1)t^m(1+O(t))\,dt,
\]
and
\[
d(v\circ z)
=
b(n+1)t^n(1+O(t))\,dt.
\]
Thus \(m=r_q(p)\) and \(n=r_z(p)\).  In the induced local
holomorphic frames of \(L_q\) and \(L_z\), the relative morphism
has coefficient
\[
g(t)
=
\frac{b(n+1)}{a(m+1)}
t^{\,n-m}(1+O(t)),
\]
and hence
\[
\ord_p(\cZ)=n-m=r_z(p)-r_q(p).
\]

\begin{proposition}
As a meromorphic section of
\[
\cH:=\Hom(L_q,L_z),
\]
the relative morphism satisfies
\[
\Div(\cZ)=R_z-R_q.
\]
\end{proposition}

\begin{proof}
The differential \(Tq\) is a holomorphic section of
\[
\Hom(TX,L_q)
\]
whose zero divisor is \(R_q\).  Similarly, the zero divisor of
\(Tz\) as a section of \(\Hom(TX,L_z)\) is \(R_z\).  Since
\[
\cZ=Tz\circ(Tq)^{-1}
\]
as a meromorphic bundle morphism, its order at every point is
\[
\ord_p(\cZ)=r_z(p)-r_q(p).
\]
This gives the stated divisor identity.
\end{proof}

\begin{remark}
The preceding statement concerns the divisor of \(\cZ\) as a
meromorphic section of the line bundle
\(\Hom(L_q,L_z)\).  It should not be confused with the divisor of
the scalar quotient \(dz/dq\) written in fixed affine
meromorphic frames.  For a meromorphic function
\(q:X\to\PP^1\), the meromorphic differential satisfies
\[
(dq)=R_q-2q^*(\infty),
\]
and similarly
\[
(dz)=R_z-2z^*(\infty).
\]
The additional terms record the divisors of the chosen
meromorphic target frames.  The bundle-theoretic identity
\(\Div(\cZ)=R_z-R_q\) is coordinate invariant.
\end{remark}

At a point where \(r_z(p)=r_q(p)\), the relative morphism may
remain regular and nonzero even if both projections are ramified.
Thus \(\cZ\) detects relative rather than absolute ramification.

\section{The logarithmic connection}

There is a small but essential distinction between a scalar
function and the relative morphism \(\cZ\).  The latter is a
meromorphic section of the line bundle
\[
\cH=\Hom(L_q,L_z),
\]
and therefore the ordinary exterior derivative \(d\cZ\) is not,
by itself, an invariantly defined derivative of a bundle-valued
section.  The notation \(\cZ^{-1}d\cZ\) must consequently be
understood through local trivializations, or equivalently as
connection data.

Let \(U_\alpha\subset X\) be an open set on which \(\cH\) has a
nonvanishing holomorphic frame \(E_\alpha\), and write
\[
\cZ=g_\alpha E_\alpha
\]
where \(g_\alpha\) is meromorphic.  Away from the zeros and poles
of \(g_\alpha\), put
\[
\alpha_\alpha
=
g_\alpha^{-1}dg_\alpha
=
d\log g_\alpha.
\]

On an overlap \(U_\alpha\cap U_\beta\), write
\[
E_\beta=h_{\alpha\beta}E_\alpha.
\]
Then
\[
g_\beta=h_{\alpha\beta}^{-1}g_\alpha,
\]
and hence
\[
\alpha_\beta
=
\alpha_\alpha-d\log h_{\alpha\beta}.
\]
This is the gauge transformation law for local logarithmic
connection data with the convention
\[
D_\alpha=d-\alpha_\alpha.
\]
Indeed, if a local coefficient transforms by
\(f_\beta=h_{\alpha\beta}^{-1}f_\alpha\), then
\[
D_\beta f_\beta
=
h_{\alpha\beta}^{-1}D_\alpha f_\alpha.
\]

\begin{definition}
The logarithmic connection determined by the relative morphism
\(\cZ\) is the meromorphic connection on
\(\cH\) whose local operators are
\[
D_\alpha=d-\alpha_\alpha,
\qquad
\alpha_\alpha=d\log g_\alpha.
\]
We use the normalized local connection form
\[
\Theta_\alpha:=2i\,\alpha_\alpha.
\]
\end{definition}

By construction,
\[
D_\alpha g_\alpha
=
dg_\alpha-\alpha_\alpha g_\alpha
=
0.
\]
Thus the relative differential morphism is parallel for the
connection that it determines.

If
\[
m_p=\ord_p(\cZ),
\]
then in a local holomorphic frame near \(p\) one may write
\[
g_\alpha=t^{m_p}u(t),
\qquad
u(0)\neq0.
\]
Therefore
\[
\Theta_\alpha
=
2i\left(
m_p\frac{dt}{t}+\frac{du}{u}
\right),
\]
so
\[
\operatorname{Res}_p\Theta_\alpha
=
2i\,m_p
=
2i\bigl(r_z(p)-r_q(p)\bigr).
\]

\section{The Maurer--Cartan interpretation}

We now identify the preceding logarithmic connection with the
Maurer--Cartan construction.

The multiplicative group \(\CC^*\) carries the Maurer--Cartan
form
\[
\vartheta_{\mathrm{MC}}=g^{-1}dg.
\]
On every open set \(U_\alpha\) where
\[
\cZ=g_\alpha E_\alpha
\]
and \(g_\alpha\neq0\), the coefficient
\[
g_\alpha:U_\alpha\longrightarrow\CC^*
\]
therefore pulls back the Maurer--Cartan form to
\[
g_\alpha^*\vartheta_{\mathrm{MC}}
=
g_\alpha^{-1}dg_\alpha.
\]

This gives a precise meaning to the symbolic expression
\[
\cZ^{-1}d\cZ.
\]
Throughout the remainder of the paper, this notation denotes the
collection of local logarithmic derivatives
\[
g_\alpha^{-1}dg_\alpha
\]
associated with local representations
\(\cZ=g_\alpha E_\alpha\).  It is not meant as the ordinary
exterior derivative of a globally defined scalar function.

In the distinguished local frames induced by local coordinates
\(q\) and \(z\), one has
\[
\cZ
=
z_q\,\frac{\partial}{\partial z}\otimes dq,
\]
so that
\[
g_\alpha=z_q
\]
and
\[
\alpha_\alpha
=
d\log z_q.
\]

\begin{proposition}[Local Maurer--Cartan description]
Let \(X\) be a Riemann surface and let
\[
q,z:X\longrightarrow\PP^1
\]
be nonconstant meromorphic maps.  Let
\[
\cZ=Tz\circ(Tq)^{-1}
\]
be their relative differential morphism, regarded as a
meromorphic section of
\[
\cH=\Hom(q^*T\PP^1,z^*T\PP^1).
\]
On the complement of \(\Div(\cZ)\), the normalized local forms
of the logarithmic connection determined by \(\cZ\) satisfy
\[
\Theta_\alpha
=
2i\,g_\alpha^{-1}dg_\alpha
\]
whenever
\[
\cZ=g_\alpha E_\alpha.
\]
Equivalently, with the preceding intrinsic convention for the
notation,
\[
\Theta
=
2i\,\cZ^{-1}d\cZ.
\]
In the coordinate frames determined by \(q\) and \(z\), this
becomes
\[
\Theta
=
2i\,d\log z_q.
\]
Thus the logarithmic connection of the relative differential morphism is locally the pullback of the Maurer--Cartan form of
\(\CC^*\).
\end{proposition}

\begin{proof}
Choose a local holomorphic frame \(E_\alpha\) of \(\cH\) and
write
\[
\cZ=g_\alpha E_\alpha.
\]
On the locus where \(g_\alpha\neq0,\infty\), the map
\(g_\alpha\) takes values in \(\CC^*\).  Pulling back the
Maurer--Cartan form of \(\CC^*\) gives
\[
g_\alpha^*(g^{-1}dg)
=
g_\alpha^{-1}dg_\alpha
=
\alpha_\alpha.
\]
By definition
\[
\Theta_\alpha=2i\,\alpha_\alpha,
\]
and hence
\[
\Theta_\alpha
=
2i\,g_\alpha^{-1}dg_\alpha.
\]
On overlaps these local forms transform by the logarithmic gauge
law established in Section~5, so they define the same
meromorphic logarithmic connection on \(\cH\).

In the coordinate frame
\[
E_\alpha=
\frac{\partial}{\partial z}\otimes dq,
\]
the coefficient of \(\cZ\) is \(g_\alpha=z_q\).  Therefore
\[
\Theta_\alpha
=
2i\,d\log z_q,
\]
which proves the final local expression.
\end{proof}

The theorem gives the intended geometric interpretation without
identifying \(\cZ\) with a globally defined \(\CC^*\)-valued
function.  The finite object is the bundle morphism \(\cZ\); its
local scalar representatives determine Maurer--Cartan forms, and
these forms glue with the gauge law of a logarithmic connection.

The construction is therefore summarized by
\[
\cZ
\quad\longmapsto\quad
\{\;g_\alpha^{-1}dg_\alpha\;\}
\quad\longmapsto\quad
\Theta.
\]
Symbolically this is precisely
\[
\Theta=2i\,\cZ^{-1}d\cZ.
\]

\section{Flatness, residues, and local reconstruction}

The Maurer--Cartan interpretation immediately gives a local
flatness statement.  In a local holomorphic frame
\(E_\alpha\) of
\[
\cH=\Hom(L_q,L_z),
\]
write
\[
\cZ=g_\alpha E_\alpha,
\qquad
\alpha_\alpha=d\log g_\alpha.
\]
On the complement of the divisor of \(\cZ\),
\[
d\alpha_\alpha=0.
\]
Since on overlaps
\[
\alpha_\beta
=
\alpha_\alpha-d\log h_{\alpha\beta},
\]
the two-forms \(d\alpha_\alpha\) agree.  Hence they define the
curvature of the logarithmic connection.

\begin{corollary}
The logarithmic connection determined by \(\cZ\) is flat on
\[
X\setminus\operatorname{supp}\Div(\cZ).
\]
Equivalently,
\[
d\Theta_\alpha=0
\]
in every local trivialization.
\end{corollary}

This is the rank-one Maurer--Cartan equation.  For
\(\CC^*\), the Lie algebra is abelian, so the quadratic term in
the usual equation
\[
d\vartheta_{\mathrm{MC}}
+
\vartheta_{\mathrm{MC}}\wedge
\vartheta_{\mathrm{MC}}
=0
\]
vanishes.

Let \(p\in X\), choose a holomorphic coordinate \(t\) centered at
\(p\), and choose a holomorphic frame \(E_\alpha\) of \(\cH\)
near \(p\).  If
\[
\cZ=t^{m_p}u(t)E_\alpha,
\qquad
u(0)\neq0,
\]
then
\[
\Theta_\alpha
=
2i\left(
m_p\frac{dt}{t}
+
\frac{du}{u}
\right).
\]
Consequently
\[
\operatorname{Res}_p\Theta_\alpha
=
2i\,m_p.
\]
Since
\[
m_p=\ord_p(\cZ)=r_z(p)-r_q(p),
\]
we obtain
\[
\operatorname{Res}_p\Theta_\alpha
=
2i\bigl(r_z(p)-r_q(p)\bigr).
\]

For a sufficiently small positively oriented loop
\(\gamma_p\) contained in this trivializing neighborhood,
\[
\int_{\gamma_p}\Theta_\alpha
=
2\pi i\,\operatorname{Res}_p\Theta_\alpha
=
-4\pi\,m_p.
\]
Thus the local logarithmic period records the order of the
relative differential morphism.

It is important to distinguish this local statement from a
claim that \(\Theta\) is a globally defined scalar one-form.
In general, the \(\Theta_\alpha\) are local connection forms and
change by exact logarithmic gauge terms.  Therefore arbitrary
global periods of a chosen scalar representative are not
intrinsic without specifying a trivialization or a corresponding
holonomy convention.

On a simply connected open set \(U\) on which \(\cH\) is
trivialized by \(E_\alpha\) and \(\cZ\) is nonzero, the relation
\[
\frac{\Theta_\alpha}{2i}
=
d\log g_\alpha
\]
can be integrated:
\[
g_\alpha
=
C_\alpha
\exp\left(
\frac{1}{2i}\int\Theta_\alpha
\right),
\qquad
C_\alpha\in\CC^*.
\]
Thus the logarithmic connection reconstructs the local scalar
representative of the relative differential morphism up to a
nonzero constant in the chosen frame.

This is the precise local inverse to the main construction:
\[
\cZ
\longmapsto
\frac{\Theta}{2i}.
\]

\section{The dual morphism and square roots}

The construction is symmetric under interchange of the two
projections.

On the complement of \(\Div(\cZ)\), define the inverse morphism
\[
\cZ^\vee:=\cZ^{-1}:L_z\longrightarrow L_q.
\]
If
\[
\cZ=g_\alpha E_\alpha,
\]
then in the dual frame \(E_\alpha^{-1}\),
\[
\cZ^\vee=g_\alpha^{-1}E_\alpha^{-1}.
\]
Therefore
\[
d\log(g_\alpha^{-1})
=
-d\log g_\alpha.
\]

\begin{proposition}
The normalized logarithmic connection forms for the inverse
relative differential morphism satisfy
\[
\Theta^\vee_\alpha=-\Theta_\alpha.
\]
In the coordinate frames induced by \(q\) and \(z\), this reads
\[
2i\,d\log q_z
=
-2i\,d\log z_q,
\qquad
q_z z_q=1.
\]
\end{proposition}

Thus inversion of the relative differential morphism corresponds
to sign reversal of its infinitesimal logarithmic form.

We next consider square roots.  The local expression
\[
\psi_\alpha=g_\alpha^{-1/2}
\]
always exists after passing to a sufficiently small simply
connected open set avoiding the divisor.  It satisfies
\[
d\log\psi_\alpha
=
-\frac12\,d\log g_\alpha
=
-\frac{\Theta_\alpha}{4i},
\]
or equivalently
\[
d\psi_\alpha
=
-\frac{\Theta_\alpha}{4i}\psi_\alpha.
\]

Globally, however, \(g_\alpha^{-1/2}\) should not automatically
be regarded as an ordinary function or as a half-differential.
The invariant question is whether there exist a line bundle \(M\)
and a meromorphic section \(\psi\) such that
\[
M^{\otimes2}\simeq\cH^{-1},
\qquad
\psi^{\otimes2}=\cZ^{-1}.
\]
Notice first that the line bundle \(\cH^{-1}\) itself always admits
some square root.  Indeed,
\[
\deg \cH^{-1}=2(\deg q-\deg z)
\]
is even, and multiplication by two on \(\operatorname{Pic}^0(X)\) is surjective.
The stronger requirement that the particular meromorphic section
\(\cZ^{-1}\) admit a meromorphic square root is controlled by the
relative ramification divisor.

\begin{proposition}[Square-root criterion]
There exist a line bundle \(M\), an isomorphism
\(M^{\otimes2}\simeq\cH^{-1}\), and a meromorphic section
\(\psi\) satisfying
\[
\psi^{\otimes2}=\cZ^{-1}
\]
if and only if the relative ramification divisor is even:
\[
R_q-R_z=2D
\]
for some integral divisor \(D\) on \(X\).  Equivalently,
\[
r_q(p)\equiv r_z(p)\pmod 2
\qquad\text{for every }p\in X.
\]
In that case one may take \(M\simeq\mathcal O(D)\), and locally
\[
\operatorname{ord}_p\psi
=
\frac{r_q(p)-r_z(p)}{2}.
\]
\end{proposition}

\begin{proof}
If \(\psi^{\otimes2}=\cZ^{-1}\), then
\[
2\Div(\psi)
=
\Div(\cZ^{-1})
=
R_q-R_z,
\]
so every coefficient of \(R_q-R_z\) is even.

Conversely, suppose
\[
R_q-R_z=2D.
\]
Since
\[
\Div(\cZ^{-1})=R_q-R_z=2D,
\]
the line bundle \(\cH^{-1}\) is isomorphic to
\(\mathcal O(2D)\).  Put \(M=\mathcal O(D)\).  Then
\[
M^{\otimes2}\simeq\mathcal O(2D)\simeq\cH^{-1}.
\]
Let \(s_D\) be a meromorphic section of \(M\) with
\(\Div(s_D)=D\).  The sections \(s_D^{\otimes2}\) and
\(\cZ^{-1}\) are meromorphic sections of the same line bundle and
have the same divisor.  Their quotient is therefore a nowhere-zero
meromorphic function on the compact Riemann surface \(X\), hence a
nonzero constant \(c\).  Choosing a square root of \(c\) and
rescaling \(s_D\) gives a section \(\psi\) satisfying
\(\psi^{\otimes2}=\cZ^{-1}\).
\end{proof}

The same criterion is visible directly in the logarithmic
connection.  At a point \(p\),
\[
\operatorname{Res}_p\Theta=2i\bigl(r_z(p)-r_q(p)\bigr),
\]
whereas
\[
\operatorname{Res}_p(d\log\psi)
=
-\frac{1}{4i}\operatorname{Res}_p\Theta
=
\frac{r_q(p)-r_z(p)}{2}.
\]
Thus a single-valued meromorphic square root requires these residues
to be integers, which is exactly the parity condition above.  Around
a small loop about \(p\), a local square root acquires the factor
\((-1)^{r_q(p)-r_z(p)}\); odd relative ramification is precisely the
local sign obstruction.

Accordingly, the notation
\[
\psi=(z_q)^{-1/2}
\]
is best understood as a local expression in the coordinate
trivialization unless the parity condition above, and hence a global
square-root structure for \(\cZ^{-1}\), has been specified.

\section{A hyperelliptic model}

We illustrate the construction on an algebraic curve while
keeping track of the relevant local frames.

Let \(X\) be the normalization of the projective closure of
\[
w^2=P(z),
\]
where \(P\) has no multiple roots.  Consider the meromorphic
projections
\[
q=w,
\qquad
z=z.
\]
On an affine region where \(q\) and \(z\) are finite,
differentiation of
\[
q^2=P(z)
\]
gives
\[
2q\,dq=P'(z)\,dz.
\]
Hence, wherever \(dq\neq0\),
\[
z_q
=
\frac{dz}{dq}
=
\frac{2q}{P'(z)}.
\]

In the corresponding affine meromorphic frames the relative
displacement is represented by
\[
\cZ
=
\frac{2q}{P'(z)}
\frac{\partial}{\partial z}\otimes dq,
\]
and its local logarithmic connection form is
\[
\Theta
=
2i\,d\log\left(\frac{2q}{P'(z)}\right).
\]
Thus
\[
\Theta
=
2i\left(
\frac{dq}{q}
-
\frac{P''(z)}{P'(z)}\,dz
\right)
\]
on the locus where this affine expression is valid.

Near a simple zero \(z=a\) of \(P\), the projection \(z\) has a
simple ramification point, whereas \(q=w\) is a local coordinate.
Indeed,
\[
z-a=cq^2+O(q^4),
\qquad
c\neq0,
\]
so
\[
z_q=2cq+O(q^3).
\]
In a holomorphic local frame of \(\cH\), \(\cZ\) therefore has a
simple zero:
\[
\ord_p(\cZ)=1=r_z(p)-r_q(p).
\]
Correspondingly,
\[
\Theta
=
2i\,\frac{dq}{q}
+
\text{holomorphic terms},
\]
and
\[
\operatorname{Res}_p\Theta=2i.
\]

Conversely, at a point with
\[
P'(z)=0,
\qquad
q\neq0,
\]
the relation
\[
2q\,dq=P'(z)\,dz
\]
shows that the projection \(q\) is ramified while \(z\) may be
regular.  The relative morphism then has a pole whose order is
the corresponding difference
\[
r_z(p)-r_q(p).
\]
This illustrates directly how the divisor
\[
\Div(\cZ)=R_z-R_q
\]
records relative ramification.

At points lying over infinity one must replace the affine target
coordinates by holomorphic coordinates near infinity.  The
bundle-theoretic formulation automatically incorporates the
resulting frame changes and is therefore preferable to treating
\(2q/P'(z)\) as a globally defined scalar representative.

\section{Relation with conformal coordinate nets}

The construction of \(\cZ\) is metric-independent.  We now show
how the logarithmic connection of \(\cZ\) is realized by the
curvature data of a pair of dual conformal coordinate nets.  We
use the sign conventions of \cite{NuramatovAC}.

Suppose locally that \(q=q_1+iq_2\) is a conformal coordinate and
\[
ds^2=H^2(dq_1^2+dq_2^2).
\]
Let \(\omega\) denote the complex curvature coefficient of this
net.  With
\[
\partial_q=\frac12(\partial_{q_1}-i\partial_{q_2}),
\]
the convention of \cite{NuramatovAC} gives
\[
\omega=-\frac{2i}{H}\,\partial_q\log H,
\qquad
H\omega\,dq=-2i\,\partial_q\log H\,dq.
\]
If \(z=z(q)\) is the dual conformal coordinate, then
\[
H=|z_q|,
\qquad
\widetilde H=H^{-1}.
\]
Let \(\widetilde\omega\) be the complex curvature coefficient of
the dual net, defined with the same convention.  Hence
\[
\widetilde\omega
=-\frac{2i}{\widetilde H}\,
\partial_z\log\widetilde H
=
2iH\,\partial_z\log H.
\]
Since
\[
\partial_z=z_q^{-1}\partial_q,
\qquad
dz=z_q\,dq,
\]
we obtain
\[
\frac1H\widetilde\omega\,dz
=
2i\,\partial_q\log H\,dq
=
-H\omega\,dq.
\]

On a simply connected neighborhood on which \(z_q\neq0\), choose
a branch of \(\log z_q\).  Holomorphicity of \(z_q\) and
\(H=|z_q|\) imply
\[
\partial_q\log H
=
\frac12\partial_q\log z_q,
\]
and therefore
\[
d\log z_q
=
2\,\partial_q\log H\,dq.
\]
Equivalently, in the notation
\(\Omega=H\omega\,dq\) of \cite{NuramatovAC}, this is precisely
the previously established reconstruction identity
\[
d\log z_q=i\Omega.
\]

\begin{theorem}[Conformal-net realization of the logarithmic connection]
Let \(q\) and \(z\) be local dual conformal coordinates related
by a holomorphic map \(z=z(q)\) with \(z_q\neq0\).  Let \(H\) be
the conformal factor of the \(q\)-net and let \(\omega\) and
\(\widetilde\omega\) be the complex curvature coefficients of the
two dual nets with the conventions above.  Then
\[
\frac1H\widetilde\omega\,dz
-
H\omega\,dq
=
2i\,d\log z_q.
\]
If \(\cZ\) denotes the relative differential morphism of the two
projections, the same identity is intrinsically expressed as
\[
\frac1H\widetilde\omega\,dz
-
H\omega\,dq
=
2i\,\cZ^{-1}d\cZ,
\]
where the right-hand side denotes the local logarithmic
Maurer--Cartan forms of \(\cZ\).
\end{theorem}

\begin{proof}
From the preceding formulas,
\[
H\omega\,dq
=-i\,d\log z_q,
\qquad
\frac1H\widetilde\omega\,dz
=+i\,d\log z_q.
\]
Subtracting the first equality from the second gives
\[
\frac1H\widetilde\omega\,dz-H\omega\,dq
=2i\,d\log z_q.
\]
In the coordinate frame of
\(\cH=\Hom(q^*T\PP^1,z^*T\PP^1)\), the local coefficient of
\(\cZ\) is \(z_q\).  Hence
\(d\log z_q=\cZ^{-1}d\cZ\) in the local logarithmic sense of
Section~6, which proves the invariant form of the identity.
\end{proof}

\begin{remark}
The order of the two terms is essential for the sign conventions
of \cite{NuramatovAC}.  With those conventions one has
\[
H\omega\,dq=-i\,d\log z_q,
\qquad
H^{-1}\widetilde\omega\,dz=+i\,d\log z_q.
\]
Thus the two dual connection differentials are opposite, and the
dual-minus-original difference is the normalized logarithmic
connection \(2i\,d\log z_q\).
\end{remark}

This theorem separates the intrinsic and metric levels of the
construction.  The meromorphic section \(\cZ\) is determined by
the two projections alone.  A choice of dual conformal-net
geometry provides a Levi--Civita realization of the logarithmic
connection already carried by \(\cZ\).

\begin{corollary}[Relative ramification from the conformal connection]
Let $p\in X$. Suppose that the dual conformal-net expression is
defined on a punctured neighborhood of $p$. Then its meromorphic
continuation is the logarithmic connection form $\Theta$, and
\[
\operatorname{Res}_p\left(\frac1H\widetilde\omega\,dz-H\omega\,dq\right)
=2i\bigl(r_z(p)-r_q(p)\bigr).
\]
Thus a nonzero residue detects a mismatch between the
ramification orders of the two projections.
\end{corollary}

\begin{proof}
On the punctured neighborhood, the main theorem identifies the
left-hand side with $2i\,d\log z_q$, interpreted as the local
logarithmic form of $\cZ$. In a holomorphic frame near $p$, write
\[
\cZ=t^{m_p}u(t)E,\qquad u(0)\neq0.
\]
Then
\[
\Theta=2i\left(m_p\frac{dt}{t}+\frac{du}{u}\right),
\]
so $\operatorname{Res}_p\Theta=2im_p$. Since
$m_p=\ord_p(\cZ)=r_z(p)-r_q(p)$, the result follows.
\end{proof}

\begin{lemma}[Frame independence of the logarithmic residue]
Let $p\in X$, let $E$ be a holomorphic frame of $\cH$ on a
neighborhood of $p$, and write $\cZ=gE$. Then
\[
\rho_p(\cZ):=\operatorname{Res}_p(d\log g)
\]
is independent of the chosen holomorphic frame and satisfies
\[
\rho_p(\cZ)=\ord_p(\cZ).
\]
Consequently the normalized residue
$\operatorname{Res}_p\Theta:=2i\,\rho_p(\cZ)$ is intrinsically
defined.
\end{lemma}

\begin{proof}
If $E'=hE$ is another holomorphic frame near $p$, then $h$ is
holomorphic and nowhere zero and the new coefficient is
$g'=h^{-1}g$. Hence
\[
d\log g'=d\log g-d\log h.
\]
Since $d\log h$ is holomorphic near $p$, it has zero residue.
Thus $\operatorname{Res}_p(d\log g)$ is frame independent. Writing
$g=t^m u(t)$ in a local coordinate $t$, with $u(0)\neq0$, gives
\[
d\log g=m\frac{dt}{t}+d\log u,
\]
so the residue is $m=\ord_p(\cZ)$.
\end{proof}

\begin{corollary}[Global degree formula]
For the logarithmic connection determined by $\cZ$, with residues
understood intrinsically as in the preceding lemma,
\[
\sum_{p\in X}\operatorname{Res}_p\Theta
=2i\,\deg\cH
=4i\bigl(\deg z-\deg q\bigr).
\]
Equivalently, whenever the conformal-net realization is available
on the complement of the singular set,
\[
\sum_{p\in X}\operatorname{Res}_p\left(\frac1H\widetilde\omega\,dz-H\omega\,dq\right)
=4i\bigl(\deg z-\deg q\bigr),
\]
where the residue at a singular point means the intrinsic residue
of the meromorphic logarithmic connection to which the displayed
local conformal form extends.
\end{corollary}

\begin{proof}
The preceding lemma and Proposition~4.1 give
\[
\sum_{p\in X}\operatorname{Res}_p\Theta
=2i\sum_{p\in X}\ord_p(\cZ)
=2i\,\deg\Div(\cZ)
=2i\bigl(\deg R_z-\deg R_q\bigr).
\]
Here no global scalar one-form is being summed: each summand is the
frame-independent local residue of the logarithmic connection.
Riemann--Hurwitz gives
\[
\deg R_q=2g-2+2\deg q,\qquad
\deg R_z=2g-2+2\deg z,
\]
and hence
\[
\deg R_z-\deg R_q=2(\deg z-\deg q).
\]
This proves the formula. Equivalently, since
\[
\cH=z^*T\PP^1\otimes(q^*T\PP^1)^{-1},
\]
one has
\[
\deg\cH=2\deg z-2\deg q
       =\deg\Div(\cZ),
\]
which gives the same identity directly.
\end{proof}

\begin{remark}
There is no contradiction with the residue theorem for a global
meromorphic one-form. The collection $\{\Theta_\alpha\}$ consists
of local connection forms of the generally nontrivial line bundle
$\cH$ and need not arise by restricting one global meromorphic
one-form on $X$. What is global here is the residue assignment
$p\mapsto\operatorname{Res}_p\Theta$, because changing a
holomorphic frame adds $-2i\,d\log h$ with $h$ holomorphic and
nowhere zero, hence with zero residue. With the convention
$D=d-\alpha$, the sign is therefore fixed by
\[
\operatorname{Res}_p\alpha=\ord_p(\cZ),\qquad
\sum_p\ord_p(\cZ)=\deg\cH.
\]
\end{remark}

\section{Discussion}

The construction may be summarized invariantly as
\[
(q,z)
\longmapsto
(L_q,L_z)
\longmapsto
\cH=\Hom(L_q,L_z)
\longmapsto
\cZ
\longmapsto
\{\alpha_\alpha=d\log g_\alpha\}.
\]
Here
\[
\cZ=g_\alpha E_\alpha
\]
in a local holomorphic frame \(E_\alpha\) of \(\cH\), and the
forms \(\alpha_\alpha\) obey the logarithmic gauge law
\[
\alpha_\beta
=
\alpha_\alpha-d\log h_{\alpha\beta}.
\]
They therefore describe a flat meromorphic logarithmic
connection away from the divisor of \(\cZ\).

The finite and infinitesimal descriptions are linked by
\[
\Theta_\alpha
=
2i\,\alpha_\alpha
=
2i\,g_\alpha^{-1}dg_\alpha.
\]
This is the precise content of the symbolic formula
\[
\Theta=2i\,\cZ^{-1}d\cZ.
\]

The divisor identity
\[
\Div(\cZ)=R_z-R_q
\]
shows that the relative differential morphism compares the
ramification of the two projections.  Locally, the same
information reappears infinitesimally as the residue
\[
\operatorname{Res}_p\Theta_\alpha
=
2i\bigl(r_z(p)-r_q(p)\bigr)
\]
in a holomorphic trivialization near \(p\).

The construction is also locally reversible.  On a simply
connected trivializing neighborhood,
\[
g_\alpha
=
C_\alpha
\exp\left(
\frac1{2i}\int\Theta_\alpha
\right).
\]
The constants and their transformations on overlaps are part of
the line-bundle data.  Thus no globally defined scalar
exponential reconstruction is asserted unless a global
trivialization has been chosen.

This distinction is essential.  The geometric object is not the
coordinate coefficient \(z_q\) alone, but the meromorphic
section \(\cZ\) of the relative homomorphism bundle.  Likewise,
the intrinsic infinitesimal object is the logarithmic connection,
not an arbitrarily chosen global scalar one-form.

\section{Conclusion}

Let
\[
q,z:X\longrightarrow\PP^1
\]
be two nonconstant meromorphic projections.  Their differentials
determine the canonical meromorphic relative differential morphism
\[
\cZ=Tz\circ(Tq)^{-1}
\]
as a section of
\[
\cH
=
\Hom(q^*T\PP^1,z^*T\PP^1).
\]
Its divisor is
\[
\Div(\cZ)=R_z-R_q,
\]
so its zeros and poles measure precisely the difference of the
ramification divisors of the two projections.

Writing locally
\[
\cZ=g_\alpha E_\alpha,
\]
the logarithmic derivatives
\[
g_\alpha^{-1}dg_\alpha
\]
are pullbacks of the Maurer--Cartan form of \(\CC^*\).  Their
gauge transformation law is exactly that of local logarithmic
connection data on \(\cH\).  With the normalization used in this
paper,
\[
\Theta_\alpha
=
2i\,g_\alpha^{-1}dg_\alpha.
\]
This gives the intrinsic meaning of the central symbolic identity
\[
\Theta=2i\,\cZ^{-1}d\cZ.
\]

In coordinate frames induced by \(q\) and \(z\), the same
identity becomes
\[
\Theta=2i\,d\log z_q.
\]
The formula therefore connects three descriptions of the same
relative geometry: ramification divisors, a meromorphic
bundle morphism, and a flat logarithmic connection.

When the projections arise from the dual conformal-net setting
of \cite{NuramatovAC}, the curvature coefficients of the two nets
satisfy
\[
\frac1H\widetilde\omega\,dz-H\omega\,dq
=
2i\,d\log z_q
=
2i\,\cZ^{-1}d\cZ.
\]
The first paper supplies the conformal reconstruction identity
\(d\log z_q=iH\omega\,dq\); the symmetric dual difference and its
bundle-theoretic interpretation are established here.  This
identifies the intrinsic logarithmic connection of the relative
differential morphism with the difference of the two dual
Levi--Civita connection differentials.  Moreover, the local residues recover the relative
ramification orders,
\[
\operatorname{Res}_p\Theta=2i\bigl(r_z(p)-r_q(p)\bigr),
\]
while their global sum satisfies
\[
\sum_{p\in X}\operatorname{Res}_p\Theta=4i\bigl(\deg z-\deg q\bigr).
\]
Thus the conformal connection data link local differential
geometry to the global ramification and degree data of the two
meromorphic projections.

\end{document}